\documentclass[twoside]{scrartcl}
\usepackage{scrlayer-scrpage}

\usepackage[utf8]{inputenc}
\usepackage[T1]{fontenc}
\usepackage[english]{babel}
\usepackage[unicode=true,plainpages=false]{hyperref}

\usepackage[left=2.5cm, right=3cm, top=2cm, bottom=3.5cm]{geometry}

\usepackage[shortlabels]{enumitem}

\usepackage{amsmath}
\usepackage{amssymb}
\usepackage{amsthm}
\usepackage{mathrsfs}
\usepackage{yhmath}
\usepackage{wasysym}
\usepackage{mathtools}

\usepackage[capitalise]{cleveref}

\usepackage{caption}
\usepackage{subcaption}

\usepackage{graphicx}
\usepackage{pgf,tikz}
\usepackage{tikz-cd}
\usepackage{rotating}

\usepackage{lmodern} 
\usepackage{tabularx} 
\usepackage[draft=false,babel,kerning=true,spacing=true]{microtype} 

\allowdisplaybreaks

\let\cprod\times

\makeatletter
\newcommand*{\relrelbarsep}{.386ex}
\newcommand*{\relrelbar}{%
  \mathrel{%
    \mathpalette\@relrelbar\relrelbarsep
  }%
}
\newcommand*{\@relrelbar}[2]{%
  \raise#2\hbox to 0pt{$\m@th#1\relbar$\hss}%
  \lower#2\hbox{$\m@th#1\relbar$}%
}
\providecommand*{\rightrightarrowsfill@}{%
  \arrowfill@\relrelbar\relrelbar\rightrightarrows
}
\providecommand*{\leftleftarrowsfill@}{%
  \arrowfill@\leftleftarrows\relrelbar\relrelbar
}
\providecommand*{\xrightrightarrows}[2][]{%
  \ext@arrow 0359\rightrightarrowsfill@{#1}{#2}%
}
\providecommand*{\xleftleftarrows}[2][]{%
  \ext@arrow 3095\leftleftarrowsfill@{#1}{#2}%
}
\makeatother

\renewcommand{\times}{\cdot}
\DeclarePairedDelimiter{\abs}{\lvert}{\rvert}

\newcommand{\noloc}{\nobreak\mskip6mu plus1mu\mathpunct{}\nonscript\mkern-\thinmuskip{:}\mskip2mu\relax} 

\newcommand{\isom}{\cong}

\newcommand{\xto}[1]{\mathbin{\xrightarrow{#1}}} 
\newcommand{\isoto}{\xto\sim}

\newcommand{\xtofrom}[1]{\mathbin{\xleftrightarrow{#1}}}
\newcommand{\isotofrom}{\xtofrom\sim}

\newcommand{\setst}{\mathrel{|}}
\newcommand{\isect}{\mathbin{\cap}}

\newcommand{\union}{\mathbin{\cup}}
\newcommand{\bigunion}{\bigcup}
\newcommand{\dunion}{\mathbin{\sqcup}}

\renewcommand{\emptyset}{\varnothing}

\renewcommand{\subset}{\subseteq}
\renewcommand{\supset}{\supseteq}

\newcommand{\comp}{\mathbin{\circ}}
\newcommand{\id}[1]{\mathrm{id}_{#1}}

\newcommand{\injto}{\mathrel{\hookrightarrow}}
\newcommand{\surjto}{\mathrel{\twoheadrightarrow}}

\newcommand{\bigdsum}{\bigoplus}

\newcommand{\Z}{\mathbb{Z}}

\DeclareMathOperator{\img}{im}

\newcommand{\tensor}{\otimes}

\DeclareMathOperator{\supp}{supp}

\newcommand{\Spec}{\mathrm{Spec}\,}
\newcommand{\MaxSpec}{\mathrm{MaxSpec}\,}

\newcommand{\red}{\mathrm{red}}

\newcommand{\catwfintype}[1]{\mathrm{wFinType}(#1)}

\newcommand{\catcoh}[1]{\mathrm{Coh}(#1)}

\DeclareMathOperator{\IHom}{\mathscr{H}\kern -3pt \mathit{om}}
\DeclareMathOperator{\IExt}{\mathscr{E}\kern -2pt \mathit{xt}}

\newcommand{\ri}{\mathcal O} 

\DeclareMathOperator{\Spa}{Spa}

\DeclareMathOperator{\Spv}{Spv}

\DeclareMathOperator{\JG}{JG} 

\theoremstyle{plain}
\newtheorem{theorem}{Theorem}[section]
\newtheorem{theorem*}{Theorem}
\newtheorem{proposition}[theorem]{Proposition}
\newtheorem{proposition*}[theorem*]{Proposition}
\newtheorem{corollary}[theorem]{Corollary}
\newtheorem{lemma}[theorem]{Lemma}

\theoremstyle{definition}
\newtheorem{definition}[theorem]{Definition}
\newtheorem{definition*}[theorem*]{Definition}

\newtheorem{remark}[theorem]{Remark}

\newtheorem{hypothesis*}[theorem*]{Hypothesis}

\numberwithin{equation}{theorem}
\numberwithin{figure}{subsection}
\numberwithin{table}{subsection}

\newlist{thmenum}{enumerate}{1}
\setlist[thmenum]{label=(\roman*), ref=\thetheorem.(\roman*)}
\crefalias{thmenumi}{theorem}

\newlist{propenum}{enumerate}{1}
\setlist[propenum]{label=(\roman*), ref=\theproposition.(\roman*)}
\crefalias{propenumi}{proposition}

\newlist{corenum}{enumerate}{1}
\setlist[corenum]{label=(\roman*), ref=\thecorollary.(\roman*)}
\crefalias{corenumi}{corollary}

\newlist{lemenum}{enumerate}{1}
\setlist[lemenum]{label=(\roman*), ref=\thelemma.(\roman*)}
\crefalias{lemenumi}{lemma}

\newlist{exampleenum}{enumerate}{1}
\setlist[exampleenum]{label=(\alph*), ref=\theexamples.(\alph*)}
\crefalias{exampleenumi}{example}

\newlist{remarksenum}{enumerate}{1}
\setlist[remarksenum]{label=(\roman*), ref=\theremarks.(\roman*)}
\crefalias{remarksenumi}{remark}

\newlist{defenum}{enumerate}{1}
\setlist[defenum]{label=(\alph*), ref=\thedefinition.(\alph*)}
\crefalias{defenumi}{definition}

\title{Normal and Irreducible Adic Spaces, the Openness of Finite Morphisms and a Stein Factorization}
\author{Lucas Mann}

\begin{document}

\maketitle

\begin{abstract}
We transfer several elementary geometric properties of rigid-analytic spaces to the world of adic spaces, more precisely to the category of adic spaces which are locally of (weakly) finite type over a non-archimedean field. This includes normality, irreducibility (in particular irreducible components) and a Stein factorization theorem. Most notably we show that finite morphisms of adic spaces are open under mild assumptions on the base and target space. \textbf{2020 MSC:} 14G22, 11G25. 
\end{abstract}

\section*{Introduction}

The goal of this paper is to provide some elementary geometric properties of adic spaces which are locally of (weakly) finite type over a field $K$, such as normality, irreducibility and a Stein factorization theorem. We apply these results, which were previously only stated in the realm of rigid-analytic and Berkovich spaces, in \cite{mann-werner-simpson} and we hope that they will turn out to be useful for other purposes as well.

While most of the desired results could easily be transferred from the rigid-analytic world to the adic world via the usual equivalences of categories, this is not the case for statements that make direct references to the underlying topological spaces. This is true for the following main result, around which the paper evolved:

\begin{theorem} \label{rslt:intro-openness-of-finite-morphism-with-loc-irr-target}
Let $K$ be a non-archimedean field, let $X$ and $Y$ be adic spaces which are locally of (weakly) finite type over $K$ and let $f\colon Y \to X$ be a finite morphism. Assume that both $X$ and $Y$ are of the same pure dimension $d$ and that $X$ is normal. Then $f$ is open.
\end{theorem}

A similar result for Berkovich spaces is found in \cite[Lemma 3.2.4]{berkovichspectral}, but \cref{rslt:intro-openness-of-finite-morphism-with-loc-irr-target} cannot be derived from it. Instead, we will use Huber's approach in \cite[Lemma 1.7.9]{huber-etale-cohomology} to reduce the claim to a similar claim about schemes, which is well-known to be true. Our proof is given in \cref{rslt:openness-of-finite-morphism-with-loc-irr-target}.

The second main result of this paper is the following adic version of the Stein Factorization Theorem:

\begin{theorem} \label{rslt:intro-stein-factorization}
Let $K$ be a non-archimedean field, let $X$ and $Y$ be adic spaces which are locally of (weakly) finite type over $K$ and let $f\colon Y \to X$ be a proper map of finite type. Then $f$ factors as
\begin{align*}
	Y \xto{h} Z \xto{g} X
\end{align*}
with the following properties:
\begin{thmenum}
	\item The map $g$ is finite.
	\item The map $h$ is proper and has geometrically connected fibers.
\end{thmenum}
\end{theorem}

While similar results can be found in the rigid-analytic setting (see \cite[Proposition 9.6.3/5]{nla.cat-vn772374}) and in Berkovich's theory (see \cite[Proposition 3.3.7]{berkovichspectral}), both of these references do not show that the fibers of $h$ are \emph{geometrically} connected. Our proof is presented in \cref{rslt:stein-factorization}.

This paper is structured as follows. In \cref{sec:normal-and-irreducible} we introduce the notions of normal and irreducible adic spaces of (weakly) finite type over a non-archimedean field $K$, thereby explaining all of the terminology occurring in \cref{rslt:intro-openness-of-finite-morphism-with-loc-irr-target}. Although one could transfer most of the definitions and results from the rigid-analytic world, we instead decided to work directly in the adic world and build everything from ground up; we found this approach more elegant while not requiring much more effort. It also allows us to work in a slightly more general setting than what the rigid-analytic world provides us with (e.g. we allow canonical compactifications). For most of the basic definitions and results we follow \cite{conrad-irr-comp-of-rigid-spaces}, with the necessary modifications to the adic world supplied where necessary. The main result of \cref{sec:normal-and-irreducible} is \cref{rslt:intro-openness-of-finite-morphism-with-loc-irr-target}.

In \cref{sec:stein-factorization} we prove the Stein Factorization \cref{rslt:intro-stein-factorization}. This is done by first transferring the rigid-analytic Stein Factorization Theorem to the adic world and then provide an additional argument (analogous to the scheme case) to show that the fibers of $h$ are \emph{geometrically} connected.

\paragraph{Notation and Conventions.} Throughout the paper, we fix a non-archimedean field $K$, i.e. $K$ is equipped with a non-trivial non-archimedean valuation under which it is complete. We let $\catwfintype K$ denote the category of adic spaces $X$ which are locally of weakly finite type over $K$ (see \cite[Definition 1.2.1.(i)]{huber-etale-cohomology}). If $X = \Spa(A, A^+)$ is affinoid then this finiteness hypothesis means precisely that $A$ is of topologically finite type over $K$ (there is no assumption on $A^+$).

\paragraph{Acknowledgments.} I am grateful to Annette Werner for helpful discussions and many comments on this paper, and to Torsten Wedhorn for suggesting to look into Lemma 1.7.9 of Huber's book.

\section{Normal and Irreducible Adic Spaces} \label{sec:normal-and-irreducible}

Let $K$ and $\catwfintype K$ be as in the conventions. We will introduce the notions of normality, irreducibility and irreducible components for objects $X \in \catwfintype K$. We also study the interactions of these notions with each other and with the purity of dimensions. Most of the ideas are taken from the rigid-analytic analogue, specifically from \cite{conrad-irr-comp-of-rigid-spaces}. At the end of the section we prove \cref{rslt:intro-openness-of-finite-morphism-with-loc-irr-target}.

The main technical idea is to reduce all questions about an adic space $X \in \catwfintype K$ to questions about the local rings at those points corresponding to maximal ideals of an affinoid covering. This idea is made more precise by the following definition and lemma.

\begin{definition}[cf. Definition 3.2 in \cite{joao-master-thesis}]
Let $X \in \catwfintype K$. The \emph{Jacobson-Gelfand spectrum} of $X$ is the subset
\begin{align*}
	\JG(X) \subset X
\end{align*}
of all rank-1 points $x \in X$ such that there is an affinoid open neighborhood $U = \Spa(A, A^+)$ of $x \in X$ with $\supp x \subset A$ being a maximal ideal.
\end{definition}

\begin{lemma} \label{rslt:properties-of-JG}
Let $X \in \catwfintype K$.
\begin{lemenum}
	\item The Jacobson-Gelfand spectrum is local, i.e. for every open subset $U \subset X$ we have $\JG(U) = \JG(X) \isect U$. Similarly, for every closed adic subspace $Z \subset X$ we have $\JG(Z) = \JG(X) \isect Z$. Moreover, $\JG(X) \subset X$ is dense.
	\item Suppose that $X = \Spa(A, A^+)$ is affinoid. Then $\supp\colon \JG(X) \isoto \MaxSpec A$ is a continuous bijection.
	\item In the setting of (ii), let $x \in \JG(X)$ with corresponding maximal ideal $\mathfrak m = \supp x \subset A$. Then $\hat\ri_{X,x} = \hat A_{\mathfrak m}$.
	\item \label{rslt:X-connected-iff-JG-connected} $X$ is connected if and only if $\JG(X)$ is connected. If $X = \Spa(A, A^+)$ is affinoid then $X$ is connected if and only if $\Spec A$ is connected.
\end{lemenum}
\end{lemma}
\begin{proof}
Clearly $K$ is a Jacobson-Tate ring in the sense of \cite[Definition 3.1]{joao-master-thesis}, hence $X$ is a Jacobson adic space by \cite[Proposition 3.3.(1)]{joao-master-thesis}. Thus (ii), (iii) and the density of $\JG(X) \subset X$ follow from \cite[Proposition 3.3.(2,3)]{joao-master-thesis}. The claim $\JG(U) = \JG(X) \isect U$ reduces easily to the case that $X$ is affinoid and $U$ is a rational open subset. Then the claim follows from (ii) and \cite[Proposition 3.3.(2)]{joao-master-thesis}. The corresponding statement for $Z$ follows easily by reducing to the affinoid case and then using (ii) and \cite[Lemma 00G9]{stacks-project}.

It remains to prove (iv). Clearly, if $\JG(X)$ is connected then so is $X$. Conversely, assume that $\JG(X) = U_1' \dunion U_2'$ for some open subsets $U_1', U_2' \subset \JG(X)$. For $i = 1, 2$ let $U_i$ be the union of all open subsets $U \subset X$ with $\JG(U) \subset U_i'$. Clearly $U_1$ and $U_2$ are disjoint open subsets of $X$, so we only have to check that they cover $X$ in order to finish the proof. Let $x \in X$ be given and let $U = \Spa(A, A^+)$ be an affinoid open neighborhood of $x$. Assume that $U$ is not contained in $U_1$ or $U_2$. Then $U_1'$ and $U_2'$ produce a disconnection of $\MaxSpec A$ and since $A$ is a Jacobson ring (by \cite[Proposition 3.3.(3)]{joao-master-thesis}) this produces a disconnection of $\Spec A$ and hence of $U$ (this also proves the second part of the claim). Then one part of this disconnection is contained in $U_1$ and the other part is contained in $U_2$, so that $x$ lies in $U_1$ or $U_2$.
\end{proof}

Apart from the Jacobson-Gelfand spectrum, the other main tool for working with $X \in \catwfintype K$ is the fact that if $X = \Spa(A, A^+)$ is affinoid then $A$ is an excellent ring by \cite[Théorème 2.13]{ducros-excellent-rings}. This allows us to make the following definition of normality. Note that in the following, a ring $A$ is called normal if all its localizations at prime ideals are domains which are integrally closed in their quotient fields; in particular, normal rings are reduced but not necessarily domains.

\begin{definition}
Let $X \in \catwfintype K$. We say that $X$ is \emph{normal} (resp. \emph{reduced}) if $X$ can be covered by affinoid adic spaces of the form $\Spa(A, A^+)$ such that $A$ is a normal (resp. reduced) ring.
\end{definition}

\begin{lemma} \label{rslt:normality-is-local}
For $X \in \catwfintype K$ the following are equivalent:
\begin{lemenum}
	\item $X$ is normal (resp. reduced).
	\item For every affinoid open subspace $\Spa(A, A^+) \injto X$, $A$ is a normal (resp. reduced) ring.
	\item For every $x \in \JG(X)$, $\hat\ri_{X,x}$ is a normal (resp. reduced) ring.
\end{lemenum}
\end{lemma}
\begin{proof}
Let $U = \Spa(A, A^+) \injto X$ be an open subspace. Then $A$ is an excellent ring by \cite[Théorème 2.13]{ducros-excellent-rings}. Thus by \cite[Lemmas 0FIZ, 0C23 and 07NZ]{stacks-project} the ring $A$ is normal (resp. reduced) if and only if for all maximal ideals $\mathfrak m \subset A$, the completed localization $\hat A_{\mathfrak m}$ is normal (resp. reduced). Now the claim follows easily from \cref{rslt:properties-of-JG}.
\end{proof}

Next we want to construct the normalization. For any ring $A$ we denote by $\tilde A$ its normalization, i.e. the integral closure of the associated reduced ring $A_\red$ inside its total ring of fractions.

\begin{lemma} \label{rslt:normalization-is-local}
Let $X = \Spa(A, A^+) \in \catwfintype K$. Let $U = \Spa(B, B^+) \injto X$ be an affinoid open subset. Then the natural map $B \to \tilde A \tensor_A B$ is a normalization.
\end{lemma}
\begin{proof}
As in the proof of \cref{rslt:normality-is-local}, $A$ is an excellent ring and in particular a Nagata ring (cf. \cite[Lemma 07QV]{stacks-project}). We can moreover replace $A$ by $A_\red$ and thus assume that $A$ is reduced. We want to apply \cite[Theorem 1.2.2]{conrad-irr-comp-of-rigid-spaces}, for which we need to verify the following three properties:
\begin{enumerate}[(a)]
	\item $B$ is flat over $A$: This is \cite[Lemma 1.7.6]{huber-etale-cohomology}.

	\item The ring $\tilde A \tensor_A B$ is normal: To see this, let $\tilde X := \Spa(\tilde A, \tilde A^+)$ and $\tilde U := \Spa(\tilde A \tensor_A B, B'^+)$, where $\tilde A^+$ is the integral closure of $A^+$ in $\tilde A$ and $B'^+$ is the integral closure of $B^+$ inside $\tilde A \tensor_A B$ (note that $A \to \tilde A$ is finite because $A$ is Nagata, see \cite[Lemma 035S]{stacks-project}). Then $\tilde U = U \cprod_X \tilde X$ is an open subset of $\tilde X$. But $\tilde X$ is normal because $\tilde A$ is normal, hence $\tilde U$ and therefore also $\tilde A \tensor_A B$ must be normal by \cref{rslt:normality-is-local}.

	\item For every minimal prime $\mathfrak p \subset A$, $B/\mathfrak p B$ is reduced. This follows by a similar argument as in (b): Letting $X' = \Spa(A/\mathfrak p, A'^+)$ and $U' = \Spa(B/\mathfrak p B, B'^+)$ (where $A'^+$ and $B'^+$ are the integral closures of $A^+$ and $B^+$) we have $U' = U \cprod_X X'$, which is an open subset of $X'$, so that the reducedness of $X'$ implies the reducedness of $U'$ by \cref{rslt:normality-is-local}.
\end{enumerate}
This finishes the proof.
\end{proof}

\begin{definition}
Let $X \in \catwfintype K$.
\begin{defenum}
	\item If $X = \Spa(A, A^+)$ is affinoid then the normalization of $X$ is the space $\tilde X := \Spa(\tilde A, \tilde A^+)$, where $\tilde A^+$ is the integral closure of (the image of) $A^+$ in $\tilde A$. 
	\item For general $X$ note that if $V \subset U$ is an inclusion of affinoid open subsets of $X$ then by \cref{rslt:normalization-is-local} there is a unique isomorphism $\tilde V \isom V \cprod_U \tilde U$ over $V$. We can thus glue the $\tilde U$'s to get an adic space $\tilde X$, called the \emph{normalization} of $X$.
\end{defenum}
\end{definition}

\begin{lemma}
Let $X \in \catwfintype K$.
\begin{lemenum}
	\item The normalization $\tilde X$ is a normal adic space and the natural projection $\tilde X \to X$ is finite surjective.
	\item Given any map $Y \to X$ from a normal adic space $Y \in \catwfintype K$, the map $Y \to X$ factors uniquely as $Y \to \tilde X \to X$.
\end{lemenum}
\end{lemma}
\begin{proof}
For (i), everything except surjectivity follows directly from the definition. To prove surjectivity it is enough to show that $\JG(\tilde X) \to \JG(X)$ is surjective because $\tilde X \to X$ is a closed map and $\JG(X) \subset X$ is dense. This reduces to the affinoid case and then to the analogous result for (affine) noetherian schemes, which is \cite[Lemma 035Q (2)]{stacks-project}.

To prove (ii), assume first that $X = \Spa(A, A^+)$ and $Y = \Spa(B, B^+)$ are affinoid. Then the claim reduces easily to the statement that every continuous ring homomorphism $A \to B$ (over $K$) factors uniquely over a continuous morphism $\tilde A \to B$. This is clear (the only issue might be the continuity, but note that any $K$-algebra morphism $\tilde A \to B$ is automatically continuous by \cite[Theorem 6.1.3/1]{nla.cat-vn772374}). If $X$ and $Y$ are general, then the uniqueness of the factorization allows us to glue the desired map from local versions, thus reducing to the affinoid case.
\end{proof}

Next we want to study dimensions:

\begin{lemma}
Let $X \in \catwfintype K$.
\begin{lemenum}
	\item \label{rslt:dim-is-sup-of-local-dim} We have
	\begin{align*}
		\dim X = \sup_{x\in \JG(X)} \dim \hat\ri_{X,x},
	\end{align*}
	where the dimension on the left is the dimension of the underlying spectral space $\abs X$ (see \cite[Definition 1.8.1]{huber-etale-cohomology}) and the dimension on the right denotes the Krull dimension.

	\item \label{rslt:dim-is-upper-semi-continuous} The function
	\begin{align*}
		\dim_X\colon \JG(X) \to \Z_{\ge 0}, \qquad x \mapsto \dim_X x := \dim \hat\ri_{X,x},
	\end{align*}
	is upper semi-continuous, i.e. for every $x \in \JG(X)$ there exists some open neighborhood $U \subset X$ of $x$ such that $\dim_X u \le \dim_X x$ for all $u \in \JG(U)$.

	\item $X$ is of pure dimension $d$ (i.e. every non-empty open subset of $X$ has dimension $d$) if and only if $\dim \hat\ri_{X,x} = d$ for all $x \in \JG(X)$.
\end{lemenum}
\end{lemma}
\begin{proof}
In all cases we can assume that $X = \Spa(A, A^+)$ is affinoid.

We start with (i). By \cite[Lemma 1.8.6]{huber-etale-cohomology} we have $\dim X = \dim A$, where $\dim A$ denotes the Krull dimension of $X$. On the one hand we have $\dim A = \sup_{\mathfrak m \in \MaxSpec A} \dim A_{\mathfrak m}$ and on the other hand $\dim A_{\mathfrak m} = \dim \hat A_{\mathfrak m} = \dim \hat\ri_{X,x}$ for all $\mathfrak m$ (see \cite[Lemma 07NV]{stacks-project}), where $x \in \JG(X)$ is the point of $X$ associated to $\mathfrak m$. This proves (i).

To prove (ii) let $x \in \JG(X)$ be given and let $\mathfrak m = \supp(x) \in \MaxSpec A$ be the corresponding maximal ideal in $A$. Let $I = \{ \mathfrak p_1, \dots, \mathfrak p_n \}$ be the collection of minimal prime ideals in $A$ and let $I_{\mathfrak m} \subset I$ be the subset of minimal prime ideals that lie inside $\mathfrak m$. For every $\mathfrak p \in I$ let $Z_{\mathfrak p} = \Spa(A/\mathfrak p, A^{\mathfrak p+})$ (where $A^{\mathfrak p+}$ is the integral closure of $A^+$ in $A/\mathfrak p$), which is a closed adic subspace of $X$ (cf. \cite[(1.4.1)]{huber-etale-cohomology}). Let $U := X \setminus \bigunion_{\mathfrak p \in I \setminus I_{\mathfrak m}} Z_{\mathfrak p}$.

Now choose any $u \in \JG(U)$. We claim that $\dim_X u \le \dim_X x$. To see this, let $\mathfrak n \subset A$ be the maximal ideal corresponding to $u$, so that $\dim_X u = \dim A_{\mathfrak n}$ (using \cite[Lemma 07NV]{stacks-project}). By the definition of $U$, $\mathfrak n = \supp u$ cannot contain any $\mathfrak p \in I \setminus I_{\mathfrak m}$, so that every maximal chain of prime ideals $\mathfrak p_0 \subsetneq \mathfrak p_1 \subsetneq \dots \subset \mathfrak n$ has to start with some $\mathfrak p_0 \in I_{\mathfrak m}$. We may thus replace $A$ by $A/\mathfrak p_0$ and assume that $A$ is an integral domain. Then \cite[Lemma 2.1.5]{conrad-irr-comp-of-rigid-spaces} implies that $A$ is equidimensional and we are done.

Part (iii) follows easily from (i) and (ii).
\end{proof}

\begin{corollary} \label{rslt:normal-connected-implies-pure-dim}
If $X \in \catwfintype K$ is normal and connected then it is pure-dimensional.
\end{corollary}
\begin{proof}
If $X = \Spa(A, A^+)$ is affinoid (so that $A$ is normal) then $A$ is an integral domain, so it is equidimensional by \cite[Lemma 2.1.5]{conrad-irr-comp-of-rigid-spaces} and hence $X$ is pure dimensional by \cref{rslt:dim-is-sup-of-local-dim}. To handle the general case, for every $d \in \Z_{\ge0}$ let $U_d \subset X$ be the union of all open subspaces which are of pure dimension $d$. From the affinoid case one deduces that the $U_d$'s form a disjoint open partition of $X$. As $X$ is connected, all but one of the $U_d$'s must be empty.
\end{proof}

\begin{lemma} \label{rslt:dim-of-finite-morphism}
Let $f\colon Y \to X$ be a finite morphism in $\catwfintype K$. Then $\dim Y \le \dim X$. If $f$ is surjective then for every $x \in \JG(X)$ there is some $y \in \JG(f^{-1}(x))$ such that $\dim_Y y = \dim_X x$; in particular $\dim Y = \dim X$.
\end{lemma}
\begin{proof}
The claim $\dim Y \le \dim X$ follows immediately from the analogous statement for schemes (see \cite[Lemma 0ECG]{stacks-project}) and \cite[Lemma 1.8.6.(ii)]{huber-etale-cohomology}. Similarly, the second claim reduces to the statement that if $A \to B$ is a finite morphism of rings such that $\Spec B \to \Spec A$ is surjective and $\mathfrak m \subset A$ is a maximal ideal then $\dim A_{\mathfrak m} = \dim (A_{\mathfrak m} \tensor_A B)$ (see e.g. \cite[Lemma 0ECG]{stacks-project}).
\end{proof}

We can finally come to the definition and study of irreducible spaces.

\begin{definition}
We say that $X \in \catwfintype K$ is \emph{irreducible} if it cannot be written as the union of two proper closed adic subspaces.
\end{definition}

\begin{lemma} \label{rslt:normal-and-connected-implies-irreducible}
If $X \in \catwfintype K$ is normal and connected then $X$ is irreducible.
\end{lemma}
\begin{proof}
It is enough to prove the following: Let $Z \subset X$ be a closed adic subspace and suppose that $Z$ contains some open subspace $U \subset X$; then $Z = X$. We can assume that both $Z$ and $U$ are connected. By \cref{rslt:normal-connected-implies-pure-dim} we know that $X$ is pure-dimensional, say $\dim X =: d$. Now let
\begin{align*}
	Z'_1 := \{ x \in \JG(Z) \setst \dim_Z x = d \}, \qquad Z'_2 := \{ x \in \JG(Z) \setst \dim_Z x < d \}.
\end{align*}
We claim that if $V \subset X$ is any affinoid connected open subset with $V \isect Z'_1 \ne \emptyset$ then $V \subset Z$. Indeed, let $V = \Spa(B, B^+)$, so that $Z \isect V = \Spa(B/J, B'^+)$ for some ideal $J \subset B$ (where $B'^+$ is the integral closure of $B^+$ in $B/J$). But then $\Spec B$ is normal and connected (cf. \cref{rslt:X-connected-iff-JG-connected}), hence irreducible, and there is some maximal ideal $\mathfrak m \subset B/J$ (the one corresponding to $x$) such that $\dim B = d = \dim (B/J)_{\mathfrak m}$. This implies $\dim B/J = \dim B$ and hence $J = (0)$, as desired.

From the previous paragraph we deduce that $Z'_1$ is open in $\JG(Z)$ and by \cref{rslt:dim-is-upper-semi-continuous} the same is true for $Z'_2$. But $\JG(Z)$ is connected by \cref{rslt:X-connected-iff-JG-connected} so that $Z'_2 = \emptyset$ because $U \subset Z$. Hence $\JG(Z) = Z'_1$ and now the previous paragraph shows that $Z$ is open in $X$. But then $Z$ is open and closed in $X$ and thus $Z = X$.
\end{proof}

\begin{lemma} \label{rslt:image-of-finite-morphism-is-adic-space}
Let $X, Y \in \catwfintype K$ and let $f\colon Y \to X$ be a finite morphism. Then $\img(f) \subset X$ is (the image of) a closed adic subspace.
\end{lemma}
\begin{proof}
We can assume that $X = \Spa(A, A^+)$ is affinoid; then $Y = \Spa(B, B^+)$ is also affinoid, $f$ is given by a finite map $f^*\colon A \to B$ and $B^+$ is the integral closure of $A^+$ in $B$. Let $J := \ker(f^*)$ and let $Z = \Spa(A/J, A'^+)$, where $A'^+$ is the integral closure of $A^+$ in $A/J$. We claim that $\img(f) = Z$. Clearly $\img(f) \subset Z$ and $\img(f)$ is closed (because $f$ is proper, e.g. by \cite[Lemma 1.4.5.(ii)]{huber-etale-cohomology}) so it is enough to show that $\img(f)$ contains all $x \in \JG(Z)$. Using \cref{rslt:properties-of-JG} this boils down to the fact that $f^*$ induces a surjective map $\MaxSpec B \surjto \MaxSpec(A/J)$ which is clear.
\end{proof}

The following definition is analogous to \cite[Definition 2.2.2]{conrad-irr-comp-of-rigid-spaces}.

\begin{definition} \label{def:irreducible-components}
Let $X \in \catwfintype K$, let $\rho\colon \tilde X \to X$ be the normalization of $X$ and let $\tilde X_i \subset \tilde X$ be the connected components of $\tilde X$. The \emph{irreducible components} of $X$ are the (reduced) closed adic subspaces $X_i := \rho(\tilde X_i) \subset X$ (cf. \cref{rslt:image-of-finite-morphism-is-adic-space}).
\end{definition}

\begin{lemma} \label{rslt:connected-components-properties}
Let $X \in \catwfintype K$ and let $(X_i)_{i\in I}$ be the irreducible components of $X$. Then the $X_i$'s are irreducible and form a locally finite cover of $X$. Moreover, for all $i \in I$ we have
\begin{align*}
	X_i \not\subset \bigunion_{i'\ne i} X_{i'}.
\end{align*}
\end{lemma}
\begin{proof}
Let $\rho\colon \tilde X \to X$ be the normalization of $X$ with connected components $(\tilde X_i)_{i\in I}$. Each $\tilde X_i$ is irreducible by \cref{rslt:normal-and-connected-implies-irreducible} and maps surjectively onto $X_i$. It follows easily that $X_i$ is irreducible (any $Z_1 \union Z_2 \subset X_i$ can be lifted to $(Z_1 \cprod_{X_i} \tilde X_i) \union (Z_2 \cprod_{X_i} \tilde X_i) = \tilde X$). The local finiteness of the cover follows from the fact that for every quasi-compact subset $U \subset X$, the preimage $\rho^{-1}(U) \subset \tilde X$ is still quasi-compact and can therefore only meet finitely many connected components.

To prove the second part of the claim let $i \in I$ be given and let $U := \tilde X_i$ and $V := \tilde X \setminus \tilde X_i$. It is enough to show that $\rho^{-1}(\rho(V))$ does not contain $U$. In fact we claim that the closed adic subspace $\rho^{-1}(\rho(V)) \isect U \subset U$ has strictly lower dimension than $U$ at all points. This can be checked locally, so we can assume that $X$ and $\tilde X = U \dunion V$ are affinoid. Then the claim follows easily from the analogous statement for schemes.
\end{proof}

\begin{corollary} \label{rslt:irreducible-equiv-normalization-connected}
Let $X \in \catwfintype K$ with normalization $\tilde X$. Then the following are equivalent:
\begin{corenum}
	\item $X$ is irreducible.
	\item $\tilde X$ is irreducible.
	\item $\tilde X$ is connected.
\end{corenum}
\end{corollary}
\begin{proof}
The equivalence of (ii) and (iii) is \cref{rslt:normal-and-connected-implies-irreducible}. As in the proof of \cref{rslt:connected-components-properties}, if $\tilde X$ is irreducible then so is $X$. Conversely, assume that $\tilde X$ is reducible, i.e. disconnected. Then $X$ has more than one irreducible component and is therefore reducible by \cref{rslt:connected-components-properties}.
\end{proof}

\begin{proposition} \label{rslt:irreducible-implies-pure-dimensional}
If $X \in \catwfintype K$ is irreducible then it is pure-dimensional.
\end{proposition}
\begin{proof}
Let $\rho\colon \tilde X \to X$ be the normalization of $X$. Then $\tilde X$ is connected by \cref{rslt:irreducible-equiv-normalization-connected}, hence pure-dimensional by \cref{rslt:normal-connected-implies-pure-dim}. But by \cref{rslt:dim-of-finite-morphism} for every $x \in \JG(X)$ there is some $\tilde x \in \JG(\tilde X)$ with $\dim_X x = \dim_{\tilde X} \tilde x = \dim \tilde X$.
\end{proof}

Most of the properties of $X \in \catwfintype K$ discussed so far are stable under étale extensions of $X$:

\begin{lemma} \label{rslt:properties-stable-under-etale-extension}
Let $X, Y \in \catwfintype K$ and let $Y \to X$ be an étale map.
\begin{lemenum}
	\item If $X$ is normal (resp. reduced) then so is $Y$.
	\item Let $\tilde X \to X$ be the normalization of $X$. Then $Y \cprod_X \tilde X$ is a normalization of $Y$.
	\item If $X$ is of pure dimension $d$ then so is $Y$.
\end{lemenum}
\end{lemma}
\begin{proof}
Since all of the discussed properties are local (see \cref{rslt:normality-is-local} for normality and reducedness), we can use \cite[Lemma 2.2.8]{huber-etale-cohomology} to reduce the claim to $X = \Spa(A, A^+)$ and $Y = \Spa(B, B^+)$  affinoid with $A \to B$ finite étale. Then (i) and (iii) follow directly from the analogous property for schemes (see e.g. \cite[Lemmas 033B, 033C, 039S]{stacks-project}). Moreover, (ii) follows from \cite[Theorem 1.2.2]{conrad-irr-comp-of-rigid-spaces}.
\end{proof}

We now come to the proof of \cref{rslt:intro-openness-of-finite-morphism-with-loc-irr-target}. The following proof essentially reduces the result to a similar result for schemes. We start with the affinoid version, from which the global result will be derived more or less formally.

\begin{lemma} \label{rslt:openness-of-finite-morphism-affinoid-integral-case}
Let $f\colon Y = \Spa(B, B^+) \to X = \Spa(A, A^+)$ be a finite morphism of affinoid noetherian adic spaces. Assume that both $A$ and $B$ are integral domains, that $A$ is normal and that the associated map $f^*\colon A \to B$ is injective. Then $f$ is open.
\end{lemma}
\begin{proof}
The argument is similar to \cite[Lemma 1.7.9]{huber-etale-cohomology}. As in the reference we define
\begin{align*}
	X' &:= \{ v \in \Spv A \setst \text{$v(a) \le 1$ for $a \in A^+$ and $v(a) < 1$ for $a \in A^{\circ\circ}$} \},\\
	Y' &:= \{ v \in \Spv B \setst \text{$v(b) \le 1$ for $b \in B^+$ and $v(b) < 1$ for $b \in B^{\circ\circ}$} \}.
\end{align*}
Moreover, by \cite[Proposition 2.6.(iii)]{huber-cts-valuations} there are retractions $r_X\colon \Spv A \to \Spv(A, A^{\circ\circ} A)$ and $r_Y\colon \Spv B \to \Spv(B, B^{\circ\circ} B)$. From \cite[Theorem 3.1]{huber-cts-valuations} we deduce that $X' \isect \Spv(A, A^{\circ\circ} A) = X$ and $Y' \isect \Spv(B, B^{\circ\circ} B) = Y$. It follows easily that by restricting $r_X$ and $r_Y$ we obtain retractions $s_X\colon X' \to X$ and $s_Y\colon Y' \to Y$. Let $f'\colon Y' \to X'$ be the restriction of the natural map $\Spv(f^*)\colon \Spv B \to \Spv A$. We obtain the following commutative diagram:
\begin{center}\begin{tikzcd}
	Y \arrow[hookrightarrow,r] \arrow[d,"f"] \arrow[rr,bend left,"\id Y"] & Y' \arrow[r,swap,"s_Y"] \arrow[d,"f'"] & Y \arrow[d,"f"]\\
	X \arrow[hookrightarrow,r] \arrow[rr,bend right,swap,"\id X"] & X' \arrow[r,"s_X"] & X
\end{tikzcd}\end{center}
By \cite[Corollary 2.1.7.(iii)]{huber-knebusch-on-valuation-spectra} the assumptions on $A$, $B$ and $f^*\colon A \to B$ are enough to guarantee that $\Spv(f^*)\colon \Spv B \to \Spv A$ is open. Note also that $B^+$ is the integral closure of $A^+$ in $B$ and $B^{\circ\circ} = A^{\circ\circ} B^+$ by finiteness of $f$ (see \cite[(1.4.2)]{huber-etale-cohomology}). This implies
\begin{align*}
	Y' &= \{ v \in \Spv B \setst \text{$v(b) \le 1$ for $b \in A^+$ and $v(b) < 1$ for $b \in A^{\circ\circ}$} \}
\end{align*}
and hence $f'^{-1}(X') = Y'$. Together with the openness of $\Spv(f^*)$ we deduce that $f'$ is open. But then the above diagram immediately implies that $f$ is open as well.
\end{proof}

\begin{theorem} \label{rslt:openness-of-finite-morphism-with-loc-irr-target}
Let $X, Y \in \catwfintype K$ and let $f\colon Y \to X$ be a finite morphism. Assume that both $X$ and $Y$ are of the same pure dimension $d$ and that $X$ is normal. Then $f$ is open.
\end{theorem}
\begin{proof}
Since the claim is local on $X$ we can assume that $X = \Spa(A, A^+)$ is affinoid and irreducible. Letting $(Y_i)_{i\in I}$ be the irreducible components of $Y$, it is enough to show that each $Y_i \to X$ is open; we can thus assume that $Y$ is irreducible as well. Since $f$ is finite and $X$ affinoid, $Y$ must be affinoid as well, say $Y = \Spa(B, B^+)$. We can furthermore assume that $A$ and $B$ are reduced, because passing to the reduction does not change the topology. Then $A$ and $B$ are integral domains of dimension $d$, the induced map $f^*\colon A \to B$ is finite and $A$ is normal. In order to apply \cref{rslt:openness-of-finite-morphism-affinoid-integral-case} it only remains to verify that $f^*\colon A \to B$ is injective. But this is clear, as otherwise $f^*$ factors over $A/I$ for some ideal $0 \ne I \subset A$, but $A/I$ has lower dimension than $A$ and the dimension of $B$ is at most the one of $A/I$; contradiction!
\end{proof}

\section{Stein Factorization} \label{sec:stein-factorization}

We will now prove the Stein Factorization Theorem for adic spaces. Again we fix a non-archimedean field $K$ and work with the category $\catwfintype K$ of adic spaces which are locally of weakly finite type over $K$ (see introduction).

The first major ingredient for the Stein Factorization Theorem is the fact that the direct image of a coherent sheaf along a proper morphism of finite type is again coherent. This has been proved by Kiehl in the rigid-analytic setting, and can easily be generalized to adic spaces in $\catwfintype K$:

\begin{lemma} \label{rslt:special-affinoid-open-immersion-implies-equiv-of-Coh}
Let $j\colon U \injto X$ be an open immersion of locally strongly noetherian analytic adic spaces such that for every open affinoid $V = \Spa(A, A^+) \subset X$ we have $U \isect V = \Spa(A, A'^+)$ for some $A'^+ \supset A^+$. Then $j_*\ri_U = \ri_X$ and the functors $j_*$ and $j^{-1}$ induce an equivalence of categories of coherent sheaves on $U$ and $X$:
\begin{align*}
	j_*\colon \catcoh{\ri_U} \isotofrom \catcoh{\ri_X}\noloc j^{-1}
\end{align*}
\end{lemma}
\begin{proof}
The claim is local on $X$, so we can assume that $X = \Spa(A, A^+)$ and $U = \Spa(A, A'^+)$ for some $A'^+ \supset A^+$. Then clearly $\ri_X(X) = A = \ri_U(U)$ and using the same reasoning for every rational open subset of $X$ shows $j_*\ri_U = \ri_X$. On the other hand by \cite[Satz 3.3.18]{huber-bewertungsspektrum-und-rigide-geometrie} (or \cite[Theorem 2.3.3]{relative-p-adic-hodge-2}) the categories $\catcoh{\ri_U}$ and $\catcoh{\ri_X}$ are both equivalent to the category of finite $A$-modules, and using $j_*\ri_U = \ri_X$ one checks easily that $j_*$ provides the required equivalence.
\end{proof}

\begin{proposition} \label{rslt:direct-image-of-coherent-along-proper-is-coherent}
Let $X, Y \in \catwfintype K$, let $f\colon Y \to X$ be a proper map of finite type and let $\mathcal M$ be a coherent sheaf of $\ri_Y$-modules on $Y$. Then $f_*\mathcal M$ is a coherent sheaf of $\ri_X$-modules on $X$.
\end{proposition}
\begin{proof}
The claim is local on $X$, so we can assume that $X = \Spa(A, A^+)$ is affinoid. Let us first consider the case $A^+ = A^\circ$. Then $X$ is quasi-separated and of finite type over $K$, hence so is $Y$. By \cite[Proposition 4.5.(iv)]{huber-generalization-of-rigid-varieties} both $X$ and $Y$ are induced from rigid-analytic varieties and the same is true for all rational subsets. Thus the claim follows from the analogous claim in rigid-analytic geometry, see \cite[Theorem 3.3]{kiehl-endlichkeitssatz}.

Now let $A^+$ be general and let $X' := \Spa(A, A^\circ)$. Then we have an open immersion $j_1\colon X' \to X$. Let $Y' := Y \cprod_X X'$, so that we also have an open immersion $j_2\colon Y' \to Y$. Letting $f'\colon Y' \to X'$ denote the restriction of $f$, we get the diagram
\begin{center}\begin{tikzcd}
	Y' \arrow[r,hookrightarrow,"j_2"] \arrow[d,"f'"] & Y \arrow[d,"f"]\\
	X' \arrow[r,hookrightarrow,"j_1"] & X
\end{tikzcd}\end{center}
Both $j_1$ and $j_2$ satisfy the hypothesis of \cref{rslt:special-affinoid-open-immersion-implies-equiv-of-Coh}, as one checks easily using the explicit construction of fiber products of adic spaces. Thus the claim follows from \cref{rslt:special-affinoid-open-immersion-implies-equiv-of-Coh} and the fact that it holds for $f'$ be the first paragraph of the proof.
\end{proof}

The second major ingredient for our Stein Factorization Theorem is the geometric connectedness of certain fibers. Our proof is analogous to the the proof for schemes.

\begin{definition}
Let $f\colon Y \to X$ be a morphism of analytic adic spaces.
\begin{defenum}
	\item Let $x\colon \Spa(k, k^+) \to X$ be a morphism, where $k$ is an analytic field with open and bounded valuation subring $k^+$. The \emph{fiber} of $f$ at $x$ is the space $Y_x := \Spa(k, k^+) \cprod_X Y$. We say that $x$ is a \emph{geometric point} of $X$ if $k$ is algebraically closed. Moreover, by abuse of notation, we will often identify a point $x \in X$ with the injection $x\colon \Spa(k(x), k(x)^+) \injto X$ (cf. \cite[Definition 1.1.7]{huber-etale-cohomology}).

	\item We say that $f$ has \emph{connected fibers} if for every $x \in X$, the fiber $Y_x$ is connected.

	\item \label{def:geometrically-connected-fibers} We say that $f$ has \emph{geometrically connected fibers} if for every geometric point $\overline x$ of $X$, the fiber $Y_{\overline x}$ is connected.
\end{defenum}
\end{definition}

\begin{lemma} \label{rslt:geometrically-connected-fibers-equiv-connected-fibers-after-etale}
Let $X, Y \in \catwfintype K$ and let $f\colon Y \to X$ be a morphism. Then the following are equivalent:
\begin{lemenum}
	\item $f$ has geometrically connected fibers.
	\item For every $x \in X$ and every $x'\colon \Spa(k, \ri_k) \to \Spa(k(x), k(x)^+) \injto X$ such that $k$ is a finite separable extension of $k(x)$, the fiber $Y_{x'}$ is connected.
	\item For every étale map $U \to X$, the base-changed morphism $f_U\colon Y \cprod_X U \to U$ has connected fibers.
\end{lemenum}
\end{lemma}
\begin{proof}
We first note that whether $f$ has connected fibers or not can be checked on the rank-1 points of $X$: If $x \in X$ is any point, let $x_0 \in X$ be the maximal generalization of $X$. Then there is a natural open immersion $Y_{x_0} \injto Y_x$ and each point of $Y_x$ has a generalization in $Y_{x_0}$. Thus every non-trivial disjoint open cover of $Y_x$ produces a non-trivial disjoint open cover of $Y_{x_0}$; therefore if $Y_{x_0}$ is connected then so is $Y_x$. We immediately deduce that (ii) implies (iii) (cf. \cite[Proposition 1.7.5]{huber-etale-cohomology}).

We next show that (iii) implies (ii): Given $x'\colon \Spa(k, \ri_k) \to X$ with image $x \in X$, it is enough to show that there is an étale map $\alpha\colon U \to X$ with a point $u \in U$ such that $x' = \alpha \comp u$. To see this we can assume that $X = \Spa(A, A^+)$ is affinoid. Then $\hat k(x)$ is the completion of the quotient field $\kappa(\mathfrak p) = \mathrm{Quot}(A/\mathfrak p)$, where $\mathfrak p = \supp x$ (see \cite[Lemma 2.4.17.(a)]{relative-p-adic-hodge-1}). By Krasner's Lemma (cf. \cite[Proposition 3.4.2/5]{nla.cat-vn772374})) there is a finite separable extension $\kappa'$ of $\kappa(\mathfrak p)$ such that $k = \kappa' \tensor_{\kappa(\mathfrak p)} k(x)$. By the proof of \cite[Lemma 00UD]{stacks-project}, after possibly replacing $A$ by the localization $A_f$ for some $f \in A$ (which corresponds to an open subspace in the adic world), there is a finite étale $A$-algebra $A'$ and a prime ideal $\mathfrak p' \subset A'$ such that $\kappa(\mathfrak p') = \kappa'$. Letting $U = \Spa(A', A'^+)$, where $A'^+$ is the integral closure of $A^+$ in $A'$, one sees easily that there is a canonical injection $u\colon \Spa(k, \ri_k) \injto U$ whose image point is the rank-1 valuation with support $\mathfrak p'$.

It is clear that (i) implies (ii), so it remains to show that (ii) implies (i). We can then assume that $X = \Spa(k, \ri_k)$ for some non-archimedean field $k$. One also easily reduces to the case that $Y = \Spa(B, B^+)$ is affinoid. By (ii), for every finite separable extension $k'$ of $k$, the space $Y \cprod_X \Spa(k', \ri_{k'}) = \Spa(B \tensor_k k', B'^+)$ is connected (where $B'^+$ is the integral closure of $B^+$ in $B \tensor_k k'$). But this means that $\Spec(B \tensor_k k')$ is connected (cf. \cref{rslt:X-connected-iff-JG-connected}), which overall implies that $\Spec B$ is geometrically connected over $k$ (see \cite[Lemma 0389]{stacks-project}). Thus, given any algebraically closed non-archimedean field $\overline k$ over $k$, $\Spec(B \tensor_k \overline k)$ is connected. Then $Y \cprod_X \Spa(\overline k, \ri_{\overline k}) = \Spa(B \tensor_k \overline k, B'^+)$ is connected by \cref{rslt:X-connected-iff-JG-connected}.
\end{proof}

\begin{lemma} \label{rslt:etale-base-change-for-coherent-sheaves}
Let
\begin{center}\begin{tikzcd}
	Y' \arrow[r,"g'"] \arrow[d,"f'"] & Y \arrow[d,"f"]\\
	X' \arrow[r,"g"] & X
\end{tikzcd}\end{center}
be a cartesian diagram of adic spaces in $\catwfintype K$ and assume that $g$ is étale and $f$ is proper of finite type. Let $\mathcal M$ be a coherent sheaf of $\ri_Y$-modules on $Y$. Then the natural base change map
\begin{align*}
	g^* f_* \mathcal M \isoto f'_* g'^* \mathcal M
\end{align*}
is an isomorphism of sheaves on $X'$.
\end{lemma}

\begin{remark} \label{rmk:flat-base-change-without-properness}
One should see \cref{rslt:etale-base-change-for-coherent-sheaves} as a weak analogue of the flat base change theorem in algebraic geometry. From the proof it is easy to see that one can weaken the hypothesis on $g$ to just being flat instead of étale (then we cannot assume that $A'$ is finite over $A$ in the proof and hence need to add completions, but noetherian completions are exact) -- we chose not to do that because we did not want to introduce flat morphisms and only need the base change result in case $g$ is étale.

It is however crucial for our proof that $f$ is proper and of finite type, because the proof (similar to the case of schemes) makes use of the fact that the direct image of a coherent sheaf along $f$ is (quasi-)coherent (by \cref{rslt:direct-image-of-coherent-along-proper-is-coherent}). In order to remove the properness assumption, one needs to have a good theory of quasi-coherent sheaves on adic spaces, which does not seem to exist in the classical language due to topological issues. Very recently, Scholze managed to fix this problem in his new theory of analytic spaces \cite{scholze-analytic-spaces}. In fact, the new formalism provides a very general base change result, see \cite[Proposition 12.14]{scholze-analytic-spaces}: The new requirement is that $f$ is qcqs and that $g$ is ``steady''; the latter condition is closely related to $g$ being adic and is therefore extremely general.
\end{remark}

\begin{proof}[Proof of \cref{rslt:etale-base-change-for-coherent-sheaves}]
The claim is local on $X'$ and it clearly holds if $g$ is an open immersion, so by using \cite[Lemma 2.2.8]{huber-etale-cohomology} we can assume that $X = \Spa(A, A^+)$ and $X' = \Spa(A', A'^+)$ are affinoid with $A \to A'$ being finite étale. By \cref{rslt:direct-image-of-coherent-along-proper-is-coherent} the sheaf $f_* \mathcal M$ is coherent, hence equals the coherent sheaf associated to the finite $A$-module $M_A := \Gamma(Y, \mathcal M)$. Then $g^* f_* \mathcal M$ is the coherent sheaf on $X'$ which is associated to the finite $A'$-module $M_A \tensor_A A'$ (cf. \cite[Lemma 01BJ]{stacks-project}). Using a similar treatment for $f'_* g'^* \mathcal M$ we see that the claim boils down to showing
\begin{align*}
	\Gamma(Y, \mathcal M) \tensor_A A' = \Gamma(Y', g'^* \mathcal M).
\end{align*}
Again by \cite[Lemma 01BJ]{stacks-project} we have $\Gamma(Y', g'^* \mathcal M) = \Gamma(Y, \mathcal M) \tensor_{\ri_Y(Y)} \ri_{Y'}(Y')$, which reduces the claim to the case $\mathcal M = \ri_Y$. Choose a finite cover $Y = \bigunion_{i = 1}^n U_i$ with all $U_i = \Spa(B_i, B_i^+)$ affinoid. For every $i = 1, \dots, n$ let $U_i' := U_i \cprod_X X' = \Spa(B_i \tensor_A A', B_i'^+)$, where $B_i'^+$ is the integral closure of $B_i^+$ in $B_i \tensor_A A'$ (cf. proof of \cite[Proposition 1.2.2]{huber-etale-cohomology}). The sheaf property of $\ri_Y$ and $\ri_{Y'}$ provides exact sequences
\begin{align*}
	0 \to \ri_Y(Y) \to \bigdsum_{i=1}^n \ri_Y(U_i) \to \bigdsum_{i,j=1}^n \ri_Y(U_i \isect U_j),\\
	0 \to \ri_{Y'}(Y') \to \bigdsum_{i=1}^n \ri_{Y'}(U_i') \to \bigdsum_{i,j=1}^n \ri_{Y'}(U_i' \isect U_j').
\end{align*}
Note that $\ri_{Y'}(U_i') = \ri_Y(U_i) \tensor_A A'$ by the explicit description of $U_i$ and $U_i'$. Moreover, since $f$ is separated all the intersections $U_i \isect U_j$ are affinoid (by the same argument as in the case of schemes, using that fiber products and closed subspaces of affinoids are affinoid; for the latter see \cite[Proposition 3.6.27]{huber-bewertungsspektrum-und-rigide-geometrie}), hence satisfy the similar relation $\ri_{Y'}(U_i' \isect U_j') = \ri_Y(U_i \isect U_j) \tensor_A A'$. Therefore the second of the above exact sequences reads
\begin{align*}
 	0 \to \ri_{Y'}(Y') \to \bigdsum_{i=1}^n \ri_Y(U_i) \tensor_A A' \to \bigdsum_{i,j=1}^n \ri_Y(U_i \isect U_j) \tensor_A A'.
\end{align*}
Comparing this to the first exact sequence and noting that $A'$ is flat over $A$, we deduce $\ri_{Y'}(Y') = \ri_Y(Y) \tensor_A A'$, as desired.
\end{proof}

\begin{lemma} \label{rslt:disconnected-fiber-implies-disconnected-neighbourhood}
Let $f\colon Y \to X$ be a proper morphism of analytic adic spaces and let $x \in X$ such that $Y_x$ is disconnected, say the union of disjoint open subsets $V_1$ and $V_2$. Then there is an open neighbourhood $U \subset X$ of $x$ such that $f^{-1}(U) = U_1 \dunion U_2$ for open subsets $U_1, U_2 \subset Y$ with $V_i \subset U_i$ for $i = 1, 2$.
\end{lemma}
\begin{proof}
We use the associated Berkovich spectrum of all the spaces involved: For any analytic adic space $Z$, let $\abs Z^B$ be the subset of all maximal points of $Z$. There is a natural continuous projection $Z \surjto \abs Z^B$ by \cite[Lemma 8.1.7.(ii)]{huber-etale-cohomology}, and if $Z$ is taut (e.g. qcqs) then $\abs Z^B$ is Hausdorff by \cite[Lemma 8.1.8.(ii)]{huber-etale-cohomology}.

The claim is local on $X$, so we can assume that $X$ is qcqs (e.g. affinoid); then $Y$ is also qcqs, so in particular $X$ and $Y$ are taut. Clearly $\abs{Y_x}^B = \abs{Y_x} \isect \abs{Y}^B$ and $\abs{Y_x}^B = \abs{V_1}^B \dunion \abs{V_2}^B$ with $\abs{V_1}^B, \abs{V_2}^B \subset \abs{Y_x}^B$ open subsets. By the discussion at the beginning of the proof, $\abs{Y}^B$, $\abs{V_1}^B$ and $\abs{V_2}^B$ are quasi-compact Hausdorff spaces and in particular T4 spaces. Hence there exist disjoint open neighborhoods $U_1''$ and $U_2''$ of $\abs{V_1}^B$ and $\abs{V_2}^B$ inside $\abs{Y}^B$. Let $U_1', U_2' \subset Y$ be their preimages under the projection $Y \to \abs{Y}^B$. Then $U_1'$ and $U_2'$ are disjoint open subsets of $Y$ such that $V_i \subset U_i'$ for $i = 1, 2$.

Let $Z := \abs Y \setminus (\abs{U_1'} \union \abs{U_2'})$, which is a closed subset of $\abs Y$. By the properness of $f$, $f(Z) \subset \abs X$ is closed. Now let $U := \abs X \setminus f(Z)$. Then $U$ is an open neighborhood of $x$ and $f^{-1}(U) = U_1 \dunion U_2$ for $U_i := f^{-1}(U) \isect U_i'$, as desired.
\end{proof}

\begin{proposition} \label{rslt:proper-direct-image-geometrically-connected-fibers}
Let $X$ and $Y$ be analytic adic spaces and let $f\colon Y \to X$ be a proper map of finite type 
such that $f_* \ri_Y = \ri_X$ (via the natural morphism). Then $f$ has geometrically connected fibers.
\end{proposition}
\begin{proof}
Let $U \to X$ be any étale morphism and let $f_U\colon Y_U := Y \cprod_X U \to U$ be the base change. By \cref{rslt:geometrically-connected-fibers-equiv-connected-fibers-after-etale} it is enough to show that $f_U$ has connected fibers. But by \cref{rslt:etale-base-change-for-coherent-sheaves} we have $f_{U*} \ri_{Y_U} = \ri_U$, hence we can replace $U$ by $X$ and reduce to showing that $f$ has connected fibers. Let $x \in X$ be given and assume that $Y_x$ is disconnected. By \cref{rslt:disconnected-fiber-implies-disconnected-neighbourhood} there is a connected open neighborhood $U \subset X$ of $x$ such that $f^{-1}(U)$ is disconnected. But then $(f_*\ri_Y)(U) = \ri_Y(f^{-1}(U))$ has a non-trivial idempotent element, while $\ri_X(U)$ does not, contradicting the hypothesis that $f_* \ri_Y = \ri_X$.
\end{proof}

We can finally prove a version of the Stein Factorization Theorem:

\begin{theorem}[Stein Factorization] \label{rslt:stein-factorization}
Let $X, Y \in \catwfintype K$ and let $f\colon Y \to X$ be a proper map of finite type. Then $f$ factors as
\begin{align*}
	Y \xto{h} Z \xto{g} X
\end{align*}
with the following properties:
\begin{thmenum}
	\item The map $g$ is finite.
	\item The map $h$ is proper and has geometrically connected fibers.
\end{thmenum}
\end{theorem}
\begin{proof}
Given any coherent sheaf $\mathcal M$ on $X$, there is a unique adic space $M$ with a (necessarily finite) map $m\colon M \to X$ such that for any affinoid $U \subset X$ with preimage $V = m^{-1}(U)$ we have $\ri_M(V) = \mathcal M(U)$. Indeed, if $X$ is affinoid this is clear, and the general case is easily obtained by glueing. Now by \cref{rslt:direct-image-of-coherent-along-proper-is-coherent} the sheaf $f_*\ri_Y$ is a coherent sheaf of $\ri_X$-modules on $X$. Applying the gluing construction to $\mathcal M = f_*\ri_Y$ we obtain a finite map $g\colon Z \to X$ with $g_*\ri_{Z} = f_*\ri_Y$. From the construction of $Z$ one sees easily that there is a map $h\colon Y \to Z$ such that $f = g \comp h$ and $h_*\ri_Y = \ri_Z$. Then \cref{rslt:proper-direct-image-geometrically-connected-fibers} implies that $h$ has geometrically connected fibers.
\end{proof}


\bibliography{bibliography}

\begin{thebibliography}{10}

\bibitem{berkovichspectral}
V.~Berkovich.
\newblock {\em Spectral Theory and Analytic Geometry Over Non-Archimedean
  Fields}.
\newblock American Mathematical Society, 1990.

\bibitem{nla.cat-vn772374}
S.~Bosch, U.~G\"{u}ntzer, and R.~Remmert.
\newblock {\em Non-{A}rchimedean analysis}, volume 261 of {\em Grundlehren der
  Mathematischen Wissenschaften [Fundamental Principles of Mathematical
  Sciences]}.
\newblock Springer-Verlag, Berlin, 1984.
\newblock A systematic approach to rigid analytic geometry.

\bibitem{conrad-irr-comp-of-rigid-spaces}
B.~Conrad.
\newblock Irreducible components of rigid spaces.
\newblock {\em Annales de l'Institut Fourier}, 49(2):473--541, 1999.

\bibitem{ducros-excellent-rings}
A.~Ducros.
\newblock Les espaces de {Berkovich} sont excellents.
\newblock {\em Annales de l’institut Fourier}, 59(4):1443--1552, 2009.

\bibitem{huber-bewertungsspektrum-und-rigide-geometrie}
R.~Huber.
\newblock {\em Bewertungsspektrum und rigide {Geometrie}}.
\newblock Number~23 in Regensburger mathematische Schriften. Fakultät für
  Mathematik der Universität Regensburg, 1993.

\bibitem{huber-cts-valuations}
R.~Huber.
\newblock Continuous valuations.
\newblock {\em Mathematische Zeitschrift}, 212(3):455--478, 1993.

\bibitem{huber-generalization-of-rigid-varieties}
R.~Huber.
\newblock A generalization of formal schemes and rigid analytic varieties.
\newblock {\em Mathematische Zeitschrift}, 217:513--551, 1994.

\bibitem{huber-etale-cohomology}
R.~Huber.
\newblock {\em {Étale Cohomology of Rigid Analytic Varieties and Adic
  Spaces}}.
\newblock Number~30 in Aspects of Mathematics. Vieweg+Teubner Verlag, 1996.

\bibitem{huber-knebusch-on-valuation-spectra}
R.~Huber and M.~Knebusch.
\newblock On valuation spectra.
\newblock {\em Contemp. Math.}, 155:167--206, 1994.

\bibitem{relative-p-adic-hodge-1}
K.~S. Kedlaya and R.~Liu.
\newblock Relative {$p$}-adic {H}odge theory: foundations.
\newblock {\em Ast\'{e}risque}, (371):239, 2015.

\bibitem{relative-p-adic-hodge-2}
K.~S. {Kedlaya} and R.~{Liu}.
\newblock {Relative p-adic Hodge theory, II: Imperfect period rings}.
\newblock {\em \url{https://arxiv.org/abs/1602.06899}}, 2016.

\bibitem{kiehl-endlichkeitssatz}
R.~Kiehl.
\newblock {Der Endlichkeitssatz für eigentliche Abbildungen in der
  nichtarchimedischen Funktionentheorie}.
\newblock {\em Inventiones mathematicae}, 2:191--214, 1966.

\bibitem{joao-master-thesis}
J.~N.~P. Lourenço.
\newblock The {Riemannian} {Hebbarkeitssätze} for pseudorigid spaces, 2017.
\newblock \url{https://arxiv.org/abs/1711.06903}.

\bibitem{mann-werner-simpson}
L.~Mann and A.~Werner.
\newblock Local systems on diamonds and $p$-adic vector bundles.
\newblock Preprint, 2020.

\bibitem{scholze-analytic-spaces}
P.~Scholze.
\newblock Lectures on {Analytic} {Geometry}, 2020.
\newblock \url{https://www.math.uni-bonn.de/people/scholze/Analytic.pdf}.

\bibitem{stacks-project}
{The Stacks Project Authors}.
\newblock \textit{Stacks Project}.
\newblock \url{http://stacks.math.columbia.edu}, 2018.

\end{thebibliography}
\addcontentsline{toc}{part}{References}

\end{document}